\documentclass[11pt,amssymb]{amsart}
\usepackage{amssymb}
\usepackage[mathscr]{eucal}

\newcommand{\Z}{{\mathbb Z}}

\newcommand{\G}{{\mathcal G}}

\newcommand{\M}{{\mathcal{M}}}

\newcommand{\m}{{\mathfrak{m}}}
\newcommand{\n}{{\mathfrak{n}}}
\newcommand{\I}{{\mathcal{I}}}
\newcommand{\cR}{{\mathcal{R}}}

\newcommand{\ba}{{\mathfrak{a}}}

\newtheorem{thm}{Theorem}[section]
\newtheorem{lemma}[thm]{Lemma}
\newtheorem{prop}[thm]{Proposition}
\newtheorem{cor}[thm]{Corollary}

\usepackage[all,cmtip]{xy}
\begin{document}
\title[Centrality of the congruence kernel]{Centrality of the congruence kernel for elementary subgroups of Chevalley groups of rank $> 1$ over noetherian rings}

\begin{abstract}
Let $G$ be a universal Chevalley-Demazure group scheme associated to a reduced irreducible root system of rank $>1.$ For a commutative ring $R$, we let $\Gamma = E(R)$ denote the elementary subgroup of the group of $R$-points $G(R).$ The congruence kernel $C(\Gamma)$ is then defined to be the kernel of the natural homomorphism $\widehat{\Gamma} \to \overline{\Gamma},$ where $\widehat{\Gamma}$ is the profinite completion of $\Gamma$ and $\overline{\Gamma}$ is the congruence completion defined by ideals of finite index. The purpose of this note is to show that for an arbitrary noetherian ring $R$ (with some minor restrictions if $G$ is of type $C_n$ or $G_2$), the congruence kernel $C(\Gamma)$ is central in $\widehat{\Gamma}.$
\end{abstract}

\author[A.S.~Rapinchuk]{Andrei S. Rapinchuk}

%\thanks{$^\flat$ Partially supported by NSF grant DMS-0502120  and the Humboldt Foundation.}

\address{Department of Mathematics, University of Virginia,
Charlottesville, VA 22904}

\email{asr3x@virginia.edu}

\author[I.A.~Rapinchuk]{Igor A. Rapinchuk}

\address{Department of Mathematics, Yale University, New Haven, CT 06502}

\email{igor.rapinchuk@yale.edu}

\maketitle

\section{Introduction}\label{S:I}

Let $G$ be a universal Chevalley-Demazure group scheme associated to a reduced irreducible root system $\Phi$ of rank $> 1$. Given a commutative ring $R$, we let $G(R)$ denote the group of $R$-points of $G$, and let $E(R) \subset G(R)$ be the corresponding elementary subgroup. (We recall that $E(R)$ is defined as the subgroup generated by the images $e_{\alpha} (R) =: U_{\alpha} (R)$ for all $\alpha \in \Phi$, where $e_{\alpha} \colon \mathbb{G}_a \to G$ is the canonical 1-parameter subgroup corresponding to a root $\alpha \in \Phi$ --- see \cite{Bo1} for details.) The goal of this note is to make a contribution to the analysis of the congruence subgroup problem for $E(R)$ over a general commutative noetherian ring $R$ (with some minor restrictions if $\Phi$ is of type $C_n$ $(n \geq 2)$ or $G_2$).

While the congruence subgroup problem for $S$-arithmetic groups is a well-established subject (see \cite{PR} for a recent survey), its analysis over general rings, at least from the point of view we adopt in this note, has been rather limited, despite a large number of results dealing with arbitrary normal subgroups of Chevalley groups over commutative rings. For this reason, we begin with a careful description of our set-up.
Let $R$ be a commutative ring and $n \geq 1.$ Then to every ideal $\mathfrak{a} \subset R$, one associates the congruence subgroup $GL_n (R, \mathfrak{a}) = \ker (GL_n (R) \to GL_n (R/ \mathfrak{a}))$, where the map is the one induced by the canonical homomorphism $R \to R/ \mathfrak{a}$. Clearly, if $\mathfrak{a}$ is of {\it finite index} (i.e. the quotient $R/ \mathfrak{a}$ is a finite ring), then $GL_n (R, \mathfrak{a})$ is a normal subgroup of $GL_n (R)$ of {\it finite index}. Given a subgroup $\Gamma \subset GL_n (R),$ we set $\Gamma (\mathfrak{a}) = \Gamma \cap GL_n (R, \mathfrak{a}).$ Then, by the congruence subgroup problem for $\Gamma$, we understand the following question:
\vskip3mm

(CSP) \hskip2mm \parbox{14.9cm}{Does every normal subgroup $\Delta \subset \Gamma$ of {\it finite index} contain the congruence subgroup $\Gamma (\mathfrak{a})$ for some ideal $\mathfrak{a} \subset R$ of {\it finite index}?}

\vskip3mm

\noindent The affirmative answer would give us information about the profinite completion $\widehat{\Gamma}$, which is precisely what is needed for the analysis of representations of $\Gamma$, as well as other issues (cf. \cite{BMS}, \cite{KN}, \cite{Sh}). However, even when $\Gamma$ is $S$-arithmetic, the answer to (CSP) is often negative. So one is instead interested in the computation of the congruence kernel, which measures the deviation from a~positive solution. For this, just as in the arithmetic case, we introduce two topologies on $\Gamma$: the profinite topology $\tau_p^{\Gamma}$ and the congruence topology $\tau_c^{\Gamma}.$ The fundamental system of neighborhoods of the identity for the former consists of all normal subgroups $N \subset \Gamma$ of finite index, and for the latter of the congruence subgroups $\Gamma (\mathfrak{a}),$ where $\mathfrak{a}$ runs through all ideals of $R$ of finite index. The corresponding completions are then given by
$$
\widehat{\Gamma} = \lim_{\longleftarrow} \Gamma / N, \ \ \ \text{where} \ N \lhd \Gamma \ \text{and} \ [\Gamma :N ] < \infty
$$
and
$$
\overline{\Gamma} = \lim_{\longleftarrow} \Gamma / \Gamma (\mathfrak{a}), \ \ \ \text{where} \ \vert R/ \mathfrak{a} \vert < \infty.
$$
As $\tau_p^{\Gamma}$ is stronger than $\tau_c^{\Gamma}$, there exists a continuous surjective homomorphism $\pi^{\Gamma} \colon \widehat{\Gamma} \to \overline{\Gamma},$ whose kernel is called the {\it congruence kernel} and denoted $C(\Gamma).$ Clearly, $C(\Gamma)$ is trivial if and only if the answer to (CSP) is affirmative; in general, its size measures the extent of deviation from the affirmative answer. Unfortunately, as remarked above, in many situations, $C(\Gamma)$ is nontrivial, and the focus of this note is on a different property, viz. the {\it centrality} of $C(\Gamma)$ (which means that $C(\Gamma)$ is contained in the center of $\widehat{\Gamma}$). We note that in some cases, centrality is almost as good as triviality (cf. \cite{KN}, \cite{Sh}), and in arithmetic cases actually implies the finiteness of $C (\Gamma).$

Returning to Chevalley groups, we observe that congruence subgroups $G(R, \mathfrak{a}) \subset G(R)$ can be defined either as pullbacks of the congruence subgroups $GL_n (R, \mathfrak{a})$ under a faithful representation of group schemes $G \hookrightarrow GL_n$ over $\Z$, or, intrinsically, as the kernel of the natural homomorphism $G(R) \to G(R/\mathfrak{a}).$

Our main result concerns the congruence kernel of the elementary group $\Gamma = E(R).$ We note that the congruence topology on $\Gamma$ is induced by that on $G(R)$, i.e. is defined by the intersections $\Gamma \cap G(R, \mathfrak{a})$, where $\mathfrak{a}$ runs over all ideals $\mathfrak{a} \subset R$ of finite index. On the other hand, the profinite topology on $\Gamma$ may {\it a priori} be different from the topology induced by the profinite topology of $G(R)$ (cf. the remarks at the end of \S 4). 

\vskip2mm

\noindent {\bf Main Theorem.} {\it Let $G$ be a universal Chevalley-Demazure group scheme corresponding to a reduced irreducible root system $\Phi$ of rank $>1.$ Furthermore, let $R$ be a noetherian commutative ring such that $2 \in R^{\times}$ if $\Phi$ is of type $C_n$ ($n \geq 2$) or $G_2$, and let $\Gamma = E(R)$ be the corresponding elementary subgroup. Then the congruence kernel $C(\Gamma)$ is central.}

\vskip2mm

The centrality of the congruence kernel for $SL_n$ ($n \geq 3$) and $Sp_{2n}$ ($n \geq 2$) over rings of algebraic integers was proved by Bass, Milnor, and Serre \cite{BMS}. Their result was generalized to arbitrary Chevalley groups of rank $> 1$ over rings of algebraic integers by Matsumoto \cite{M1}. The only known result for general rings is due to Kassabov and Nikolov \cite{KN}, where centrality was established for $SL_n (\Z [x_1, \dots, x_k])$, with $n \geq 3$, and hence for the elementary group $E_n (R)$ over any finitely generated ring $R$, using $K$-theoretic methods. Although our proof shares some elements with the argument in \cite{KN}, it is purely group-theoretic and is inspired by the proof of centrality for $SL_n$ ($n \geq 3$) over arithmetic rings given in \cite{AR}; in addition, we do not use any results of Matsumoto \cite{M1}.

\vskip2mm

\noindent {\bf Conventions and notations.} All of our rings will be assumed to be commutative and unital. Unless explicitly stated otherwise, $G$ will always denote a universal Chevalley-Demazure group scheme corresponding to a reduced irreducible root system $\Phi$ of rank $> 1$. Furthermore, if $R$ is a commutative ring, then for a subgroup $\Gamma \subset G(R)$, we let $\widehat{\Gamma}$ and $\overline{\Gamma}$ denote the profinite and congruence completions of $\Gamma$, respectively. 

\section{Structure of $\overline{G(R)}$}

Let $\mathcal{I}$ be the set of all ideals $\mathfrak{a} \subset R$ of finite index, and let $\mathcal{M} \subset \mathcal{I}$ be the subset of maximal ideals. It is not difficult to see (cf. the proof of Proposition \ref{P-2}) that $\overline{G(R)}$ can be identified with the closure of the image of $G(R)$ in $G(\widehat{R})$, where
$$
\widehat{R} = {\lim_{\longleftarrow}}_{\mathfrak{a} \in \I} R/ \mathfrak{a}
$$
is the profinite completion of $R$. The proof of the Main Theorem relies on the fact that $G(\widehat{R})$ has the bounded generation property with respect to the set $\widehat{S} = \{ e_{\alpha} (t) \mid t \in \widehat{R}, \ \alpha \in \Phi \}$ of elementaries, which we will establish at the end of this section (cf. Corollary \ref{C-1}). We begin, however, by describing the structure of $\widehat{R}$ itself. For each $\mathfrak{m} \in \M$, we let $$R_{\mathfrak{m}} = \lim_{\longleftarrow} R/ \mathfrak{m}^n$$ denote the $\m$-adic completion of $R$ (cf. \cite{At}, Chapter 10).

\begin{lemma}\label{L-1}
Let $R$ be a noetherian ring.

\vskip1mm

\noindent {\rm (1)} \parbox[t]{15cm}{There exists a natural isomorphism of topological rings
$$
\widehat{R} = \prod_{\m \in \M} R_{\m}.
$$}

\vskip1mm

\noindent {\rm (2)} \parbox[t]{15cm}{Each $R_{\m}$ is a complete local ring.}
\end{lemma}
\begin{proof}
(1) Since $R$ is noetherian, for any $\mathfrak{a} \in \I$ and any $n \geq 2$, the quotient $\mathfrak{a}^{n-1}/ \mathfrak{a}^n$ is a finitely generated $R/ \mathfrak{a}$-module, hence finite. It follows that $R/ \mathfrak{a}^n$ is finite for any $n \geq 1.$ In particular, for any $\m \in \M$ and $n \geq 1,$ there exists a natural continuous surjective projection
$$
\rho_{\m, n} \colon \widehat{R} \to R/ \m^n.
$$
For a fixed $\m$, the inverse limit of the $\rho_{\m, n}$ over all $n \geq 1$ yields a continuous ring homomorphism $\rho_{\m} \colon \widehat{R} \to R_{\m}.$ Taking the direct product of the $\rho_{\m}$ over all $\m \in \M,$ we obtain a continuous ring homomorphism
$$
\rho \colon \widehat{R} \to \prod_{\m \in \M} R_{\m} =: \overline{R}.
$$
We claim that $\rho$ is the required isomorphism.

Note that ideals of the form
$$
\overline{\ba} = \m_1^{\alpha_1} R_{\m_1} \times \cdots \times \m_n^{\alpha_n} R_{\m_n} \times \prod_{\m \neq \m_i} R_{\m},
$$
where $\{ \m_1, \dots, \m_n \} \subset \M$ is a finite subset and $\alpha_i \geq 1,$ form a base of neighborhoods of zero in $\overline{R}$, with
$$
\overline{R}/ \overline{\ba} = R / \m_1^{\alpha_1} \times \cdots \times R/ \m_{n}^{\alpha_n}
$$
(cf. \cite{At}, Proposition 10.15). Set $\ba = \m_1^{\alpha_1} \cdots \m_n^{\alpha_n}.$ By the Chinese Remainder Theorem,
$$
R / \ba \simeq R/ \m_1^{\alpha_1} \times \cdots \times R/ \m_n^{\alpha_n},
$$
which implies that the composite map
$$
\widehat{R} \to \overline{R} \to \overline{R}/ \overline{\ba}
$$
is surjective. Since this is true for all $\overline{\ba},$ we conclude that the image of $\rho$ is dense. On the other hand, $\widehat{R}$ is compact, so the image is closed, and we obtain that $\rho$ is in fact surjective.

To prove the injectivity of $\rho$, we observe that for any $\ba \in \I$, the quotient $R/ \ba$, being a finite, hence artinian ring, is a product of finite local ring $R_1, \dots, R_r$ (\cite{At}, Theorem 8.7). Furthermore, for each maximal ideal $\n_i \subset R_i$, there exists $\beta_i \geq 1$ such that $\n_i^{\beta_i} = 0$ (cf. \cite{At}, Proposition 8.4). Letting $\m_i$ denote the pullback of $\n_i$ in $R$, we obtain that $\ba$ contains $\mathfrak{b} := \m_1^{\beta_1} \cdots \m_r^{\beta_r} \in \I.$ It follows that any nonzero $x \in \widehat{R}$ will have a nonzero projection to some $R/ \mathfrak{b} = R/ \m_1^{\beta_1} \times \cdots \times R/ \m_r^{\beta_r}$, and hence to some $R_{\m_i}$, as required.

\vskip1mm

\noindent (2) It is well-known that $R_{\m}$ is both complete and local (cf. \cite{At}, Propositions 10.5 and 10.16).
\end{proof}

As a first step towards establishing bounded generation of $G(\widehat{R})$ with respect to the set of elementaries, we prove

\begin{prop}\label{P-1}
There exists an integer $N = N(\Phi)$, depending only on the root system $\Phi$, such that for any commutative local ring $\cR$, any element of $G(\cR)$ is a product of $\leq N$ elements of $S = \{ e_{\alpha} (r) \mid r \in \cR, \ \alpha \in \Phi \}.$
\end{prop}
\begin{proof}
Fix a system of simple roots $\Pi \subset \Phi,$ and let $\Phi^+$ and $\Phi^-$ be the corresponding sets of positive and negative roots. Let $T \subset G$ be the canonical maximal torus, and $U^+$ and $U^-$ be the canonical unipotent $\Z$-subschemes corresponding to $\Phi^+$ and $\Phi^-.$ It is well-known (see, for example, \cite{Bo1}, Lemma 4.5) that the product map $\mu \colon U^- \times T \times U^+ \to G$ is an isomorphism onto a principal open subscheme $\Omega \subset G$ defined by some $d \in \Z[G].$ We have decompositions
$$
U^{\pm} = \prod_{\alpha \in \Phi^{\pm}} U_{\alpha} \ \ \ \text{and} \ \ \ T = \prod_{\alpha \in \Pi} T_{\alpha},
$$
where $T_{\alpha}$ is the maximal diagonal torus in $G_{\alpha} = <U_{\alpha}, U_{-\alpha}> = SL_2.$ So, the identity
$$
\left( \begin{array}{cl} a & 0 \\ 0 & a^{-1} \end{array} \right) = \left( \begin{array}{lr} 1 & -1 \\ 0 & 1 \end{array} \right) \left( \begin{array}{cc} 1 & 0 \\ 1-a & 1 \end{array} \right) \left( \begin{array}{ll} 1 & a^{-1} \\ 0 & 1 \end{array} \right) \left( \begin{array}{cc} 1 & 0 \\ a(a-1) & 1 \end{array} \right)
$$
shows that there exists $N_1 = N_1 (\Phi)$ such that any element of $\Omega (R)$ is a product of $\leq N_1$ elementaries, for {\it any} ring $\cR.$

On the other hand, it follows from the existence of the Bruhat decomposition in Chevalley groups over fields that there exists $N_2 = N_2 (\Phi)$ such that any element of $G(k)$ is a product of $\leq N_2$ elementaries, for any field $k.$ We will now show that $N:= N_1 + N_2$ has the required property for any local ring $\cR.$ Indeed, let $\m \subset \cR$ be the maximal ideal, and $k = \cR/ \m$ be the residue field. As $G(k)$ is generated by elementaries, the canonical homomorphism $\omega \colon G(\cR) \to G(k)$ is surjective. Given $g \in G(\cR)$, there exists $h \in G(\cR)$ that is a product of $\leq N_2$ elementaries and for which we have $\omega (g) = \omega (h).$ Then, for $t = gh^{-1}$, we have $\omega (t) = 1$ (in particular, $\omega (t) \in \Omega (k)$), and therefore $d(t) \not\equiv 0 (\text{mod} \ \m)$. Since $\cR$ is local, this means that $d(t) \in \cR^{\times}$, and therefore $t \in \Omega (\cR)$. Thus, $t$ is a product of $\leq N_1$ elementaries, and the required fact follows.
\end{proof}

Next, we have the following
\begin{lemma}\label{L-2}
Let $\cR_i$ ($i \in I$) be a family of commutative rings such that there exists an integer $N$ with the property that for any $i \in I,$ any $x_i \in G(\cR_i)$ is a product of $\leq N$ elementaries. Set $\cR = \prod_{i \in I} \cR_i.$ Then any $x \in G(\cR)$ is a product of $\leq N \cdot \vert \Phi \vert$ elementaries.
\end{lemma}
\begin{proof}
It is enough to observe that any element of the form
$$
(e_{\alpha_i} (r_i)) \in G(\cR) = \prod_{i \in I} G(\cR_i),
$$
with $\alpha_i \in \Phi,$ $r_i \in \cR_i$, can be written as
$$
\prod_{\alpha \in \Phi} e_{\alpha} (t_{\alpha})
$$
for some $t_{\alpha} \in \cR.$
\end{proof}
Using this result, together with Lemma \ref{L-1} and Proposition \ref{P-1}, we obtain
\begin{cor}\label{C-1}
Let $R$ be a noetherian ring. Then there exists an integer $M > 0$ such that any element of $G(\widehat{R})$ is a product of $\leq M$ elementaries from the set $\widehat{S} = \{e_{\alpha} (t) \mid t \in \widehat{R}, \alpha \in \Phi \}.$
\end{cor}
As we noted earlier, one can identify the congruence completion $\overline{G(R)}$ with the closure of the image of $G(R)$ in $G(\widehat{R})$. The following proposition gives more precise information.

\begin{prop}\label{P-2}
Let $R$ be a noetherian ring. Then $\overline{E(R)} = \overline{G(R)}$ can be naturally identified with $G(\widehat{R})$. Furthermore, there exists an integer $M > 0$ such that any element of $\overline{E(R)} = \overline{G(R)}$ is a product of $\leq M$ elements of the set $\overline{S} := \overline{ \{ e_{\alpha} (r) \mid \alpha \in \Phi, r \in R \} } $ (closure in the congruence topology).
\end{prop}
\begin{proof}
For any $\mathfrak{a} \in \I$, there exists a natural injective homomorphism $\omega_{\mathfrak{a}} \colon G(R) / G(R, \mathfrak{a}) \to G(R/ \mathfrak{a}),$ where as before, $G(R, \mathfrak{a})$ is the principal congruence subgroup of level $\mathfrak{a}.$ Taking the inverse limit over all $\mathfrak{a} \in \I,$ we obtain a continuous injective homomorphism
$$
\omega \colon \overline{G(R)} \to G(\widehat{R}).
$$
Clearly, the image of $\omega$ coincides with the closure of the image of the natural homomorphism $G(R) \to G(\widehat{R})$. 
From the definitions, one easily sees that if $\overline{e_{\alpha} (r)}$ is the image of $e_{\alpha} (r)$ ($\alpha \in \Phi, r \in R$) in $\overline{G(R)}$, then
$$
\omega (\overline{e_{\alpha} (r)}) = e_{\alpha} (\hat{r}),
$$
where $\hat{r}$ is the image of $r$ in $\widehat{R}.$ It follows that $\omega$ maps $\overline{S}$ onto $\widehat{S} = \{ e_{\alpha} (t) \mid \alpha \in \Phi, \ t \in \widehat{R} \}.$ Since by Corollary \ref{C-1}, $\widehat{S}$ generates $G(\widehat{R}),$ we obtain that $\omega (\overline{E(R)}) = G(\widehat{R})$, and consequently $\omega$ identifies $\overline{E(R)} = \overline{G(R)}$ with $G(\widehat{R})$. Furthermore, if $M$ is the same integer as in Corollary \ref{C-1}, then since every element of $G(\widehat{R})$ is a product of $\leq M$ elements of $\widehat{S}$, our second claim follows.
\end{proof}

\vskip1mm

\noindent {\bf Remark.} Recall that a group $\G$ is said to have {\it bounded generation} with respect to a generating set $X \subset \G$ if there exists an integer $N > 0$ such that every $g \in \G$ can be written as $g = x_1^{\varepsilon_1} \cdots x_d^{\varepsilon_d}$ with $x_i \in X$, $d \leq N$, and $\varepsilon_i = \pm 1.$ It follows from the Baire category theorem (cf. \cite{Mun}, Theorem 48.2) that if a compact topological group $\G$ is (algebraically) generated by a compact subset $X$, then in fact, $\G$ is automatically {\it boundedly} generated by $X$. Indeed, replacing $X$ by $X \cup X^{-1} \cup \{1 \},$ we may assume that $X = X^{-1}$ and $1 \in X.$ Set $X^{(n)} = X \cdots X$ ($n$-fold product). Then the fact that $\G = < X>$ means that
$$
\G = \bigcup_{n \geq 1} X^{(n)}.
$$
Since each $X^{(n)}$ is compact, hence closed, we conclude from Baire's theorem that for some $n \geq 1,$ $X^{(n)}$ contains an open set. Then $\G$ can be covered by finitely many translates of $X^{(n)}$, and therefore there exists $M > 0$ such that $X^{(M)} = \G$, as required. This remark shows, in particular, that (algebraic) generation of $\overline{G (R)}$ by $\overline{S}$, or that of $G(\widehat{R})$ by $\widehat{S},$ automatically yields bounded generation.

\vskip2mm

We would like to point out that the fact that $\overline{G(R)} = \overline{E(R)}$ is not used in the proof of the Main Theorem; all we need is that $\overline{E(R)}$ is boundedly generated by $\overline{S}.$ So, we will indicate another way to prove this based on some ideas of Tavgen (cf. \cite{Tav}, Lemma 1), which also gives an explicit bound on the constant $M$ in Proposition \ref{P-2}. First we observe that it is enough to establish the bounded generation of $E(\widehat{R})$ by $\widehat{S} = \{ e_{\alpha} (t) \mid \alpha \in \Phi, \ t \in \widehat{R} \}$ (indeed, this will show that $E(\widehat{R})$ is a continuous image of $\widehat{R}^N$ for some $N > 0$, hence compact, implying that the map $\omega$ from the proof of Proposition \ref{P-2} identifies $\overline{E(R)}$ with $E(\widehat{R})$, and also $\overline{S}$ with $\widehat{S}$). In turn, by the same argument as above, we see that to prove bounded generation of $E(\widehat{R})$, it suffices to show that there exists an integer $N >0$ depending only on $\Phi$ such that for any local ring $R$, any element of $E(R)$ is a product of $\leq N$ elementaries. We will show that in fact
\begin{equation}\label{E-BG}
E(R) = (U^+ (R) U^- (R))^4,
\end{equation}
so one can take $N = 4 \cdot \vert \Phi \vert.$ Let us now prove (\ref{E-BG}) by induction on the rank $\ell$ of $\Phi$. If $\ell = 1,$ then $G = SL_2$, and one easily checks that $$G(R) = E(R) = (U^+ (R) U^- (R))^4.$$ Now, we assume that (\ref{E-BG}) is valid for every reduced irreducible root system of rank $\leq \ell-1$, with $\ell \geq 2$, and prove it for a root system $\Phi$ of rank $\ell.$ Set $X = (U^+(R) U^-(R))^4$, and let $\Delta \subset \Phi$ be a system of simple roots. Since the group $E(R)$ is generated by $e_{\pm \beta} (t)$ for $\beta \in \Delta$ and $t \in R$ (cf. the proof of (\ref{E:St-1}) in \S 4), to prove (\ref{E-BG}), it suffices to show that
$$
e_{\pm \beta} (t) X \subset X.
$$
Pick $\alpha \in \Delta,$ $\alpha \neq \beta$, that corresponds to an extremal node in the Dynkin diagram of $\Phi.$ Let $\Phi_0$ (resp., $\Phi_1$) be the set of roots in $\Phi$ that do not contain (resp., contain) $\alpha$, and let $\Phi_i^{\pm} = \Phi_i \cap \Phi^{\pm}.$ Then $\Phi_0$ is an irreducible root system having $\Delta_0 = \Delta \setminus \{ \alpha \}$ as a system of simple roots; in particular, $\Phi_0$ has rank $\ell - 1.$ If we let $G_0$ denote the corresponding universal Chevalley-Demazure group scheme, then by the induction hypothesis
$$
E_0 (R) = (U_0^+ (R) U_0^- (R))^4,
$$
with the obvious notations. Let $U_1^{\pm} (R)$ be the subgroup generated by $e_{\alpha} (r)$ for $\alpha \in \Phi_1^+$ (resp., $\alpha \in \Phi_1^-$) and $r \in R.$ Then
$U^{\pm} (R) = U_0^{\pm} (R) U_1^{\pm} (R)$, and according to (\cite{Stb}, Lemma 17),
$$
U_0^{\pm} (R) U_1^{\mp} (R) = U_1^{\mp} (R) U_0^{\pm} (R).
$$
So,
$$
X = (U_0^+ (R) U_1^+ (R) U_0^- (R) U_1^- (R))^4 = (U_0^+ (R) U_0^- (R))^4 (U_1^+ (R) U_1^- (R))^4 = E_0 (R) (U_1^+ (R) U_1^- (R))^4.
$$
Since $e_{\pm \beta} (t) \in E_0 (R),$ we obtain that
$$
e_{\pm \beta} (t) X = e_{\pm \beta} (t) E_0 (R) (U_1^+ (R) U_1^- (R))^4 = X,
$$
as required.

\section{Profinite and congruence topologies coincide on 1-parameter root subgroups}

\begin{prop}\label{P-3}
Let $\Phi$ be a reduced irreducible root system of rank $\geq 2$, $G$ be the corresponding universal Chevalley-Demazure group scheme, and $E(R)$ be the elementary subgroup of the group $G(R)$ over a commutative ring $R$. Furthermore, suppose $N \subset E(R)$ is a normal subgroup of finite index. If $\Phi$ is not of type $C_n$ $(n \geq 2)$ or $G_2$, then there exists an ideal $\ba \subset R$ of finite index such that \begin{equation}\label{E-Ideal}
e_{\alpha} (\ba) \subset N \cap U_{\alpha} (R)
\end{equation}
for all $\alpha \in \Phi$, where $e_{\alpha} (\ba) = \{ e_{\alpha} (t) \mid t \in \ba \}.$ The same conclusion holds for $\Phi$ of type $C_n$ $(n \geq 2)$ and $G_2$ if $2 \in R^{\times}$. Thus, in these cases, the profinite and the congruence topologies of $E(R)$ induce the same topology on $U_{\alpha} (R)$, for all $\alpha \in \Phi.$
\end{prop}
\begin{proof}
We begin with two preliminary remarks. First, for any root $\alpha \in \Phi$,
$$
\ba (\alpha) := \{ t \in R \mid e_{\alpha} (t) \in N \}
$$
is obviously a finite index subgroup of the additive group of $R$. What one needs to show is that either $\ba(\alpha)$ itself is an ideal of $R$, or that it at least contains an ideal of finite index. Second, if $\alpha_1, \alpha_2 \in \Phi$ are roots of the same length, then by (\cite{H1}, 10.4, Lemma C), there exists an element $\tilde{w}$ of the Weyl group $W(\Phi)$ such that $\alpha_2 = \tilde{w} \cdot \alpha_1$. Consequently, it follows from (\cite{St1}, 3.8, relation (R4)) that we can find $w \in E(R)$ such that
$$
w e_{\alpha_1}(t) w^{-1} = e_{\alpha_2} (\varepsilon(w) t)
$$
for all $t \in R$, where $\varepsilon(w) \in \{ \pm 1 \}$ is independent of $t.$ Since $N$ is a normal subgroup of $E(R)$, we conclude that
\begin{equation}\label{E-Ideal1}
\ba (\alpha_1) = \ba (\alpha_2).
\end{equation}
Thus, it is enough to find a finite index ideal $\ba \subset R$ such that (\ref{E-Ideal}) holds for a {\it single} root of each length.

Let us now prove our claim for $\Phi$ of type $A_2$ using explicit computations with commutator relations. We will use the standard realization of $\Phi$, described in \cite{Bour}, where the roots are of the form $\varepsilon_i - \varepsilon_j$,
with $i,j \in \{1, 2, 3 \}, i \neq j.$ To simplify notation, we will write $e_{ij} (t)$ to denote $e_{\alpha} (t)$ for $\alpha = \varepsilon_i - \varepsilon_j.$ Set $\alpha_1 = \varepsilon_1 - \varepsilon_2.$ We will now show that $\ba (\alpha_1)$ is an ideal of $R$, and then it will follow from our previous remarks that $\ba := \ba(\alpha_1)$ is as required. Let $r \in \ba (\alpha_1)$ and $s \in R.$ Since $N \lhd E(R)$, the (well-known) relation
$$
[e_{12} (r), e_{23} (s)] = e_{13} (rs),
$$
where $[g,h] = gh g^{-1} h^{-1}$, shows that $rs \in \ba (\alpha_2)$ for $\alpha_2 = \varepsilon_1 - \varepsilon_2.$ But then (\ref{E-Ideal1}) yields $rs \in \ba(\alpha_1),$ completing the argument.

Now let $\Phi$ be any root system of rank $\geq 2$ in which all roots have the same length. Then clearly $\Phi$ contains a subsystem $\Phi_0$ of type $A_2$, so our previous considerations show that there exists a~finite index ideal $\ba \subset R$ with the property that $\ba \subset \ba(\alpha)$ for all $\alpha \in \Phi_0.$ But then, by (\ref{E-Ideal1}), the same inclusion holds for all $\alpha \in \Phi.$

Next, we consider the case of $\Phi$ of type $B_n$ with $n \geq 3$. Note that since the system of type $F_4$ contains a subsystem of type $B_3$, this will automatically take care of the case when $\Phi$ is of type $F_4$ as well. 
We will use the standard realization of $\Phi$ of type $B_n$, where the roots are of the form $\pm \varepsilon_i$, $\pm \varepsilon_i \pm \varepsilon_j$ with $i, j \in \{ 1, \dots, n \}, i \neq j.$ The system $\Phi$ contains a subsystem $\Phi_0$ of type $A_{n-1}$, all of whose roots are long roots in $\Phi.$ Arguing as above, we see that there exists an ideal $\ba \subset R$ of finite index such that (\ref{E-Ideal}) holds for all $\alpha \in \Phi_0,$ and hence for all long roots $\alpha \in \Phi.$ To show that the same ideal also works for short roots, we will use the following relation, which is verified by direct computation:
\begin{equation}\label{E-3}
[e_{\varepsilon_1 + \varepsilon_2} (r), e_{-\varepsilon_2}(s)] = e_{\varepsilon_1} (rs) e_{-\varepsilon_1 - \varepsilon_2} (-rs^2).
\end{equation}
for any $r, s \in R$. Now, if $r \in \ba,$ then $e_{\varepsilon_1 + \varepsilon_2} (r), e_{- \varepsilon_1 - \varepsilon_2} (-r) \in N.$ So, setting $s = 1$ in (\ref{E-3}) and noting that $[e_{\varepsilon_1 + \varepsilon_2} (r), e_{- \varepsilon_2} (1)] \in N$ as $N \lhd E(R),$ we obtain that $e_{\varepsilon_1} (r) \in N.$ Thus, (\ref{E-Ideal}) holds for $\alpha = \varepsilon_1$, and therefore for all short roots.

Next, we proceed to the case of $\Phi$ of type $B_2 = C_2$, where we assume that $2 \in R^{\times}.$ 
We will use the same realization of $\Phi$ as in the previous paragraph (for $n = 2$). Set $\ba = \ba (\varepsilon_1).$ Then for $r \in \ba$, $s \in R$, one can check by direct computation that
\begin{equation}\label{E-2}
[e_{\varepsilon_1} (r), e_{\varepsilon_2} (s/4)] = e_{\varepsilon_1 + \varepsilon_2} (rs/2) \in N.
\end{equation}
Next, using (\ref{E-3}), in conjunction with the fact that $e_{\varepsilon_1} (u)$ and $e_{\varepsilon_1 - \varepsilon_2} (v)$ commute for all $u, v \in R$, we obtain
$$
[e_{\varepsilon_1 + \varepsilon_2} (rs/2), e_{- \varepsilon_2} (1)][e_{\varepsilon_1 + \varepsilon_2} (rs/2), e_{- \varepsilon_2} (-1)]^{-1} = e_{\varepsilon_1} (rs) \in N,
$$
i.e. $rs \in \ba$, which shows that $\ba$ is an ideal. Furthermore, from (\ref{E-2}), we see that for any $r \in \ba,$ we have
$$
[e_{\varepsilon_1} (r), e_{\varepsilon_2} (1/2)] = e_{\varepsilon_1 + \varepsilon_2} (r) \in N.
$$
Thus, $e_{\varepsilon_1 + \varepsilon_2} (\ba) \subset N,$ and therefore (\ref{E-Ideal}) holds for all $\alpha \in \Phi.$

Finally, suppose that $\Phi$ is of type $G_2$ and assume again that $2 \in R^{\times}.$
We will use the realization of $\Phi$ described in \cite{CK}: one picks a system of simple roots $\{k, c \}$ in $\Phi$, where $k$ is long and $c$ is short, and then the long roots of $\Phi$ are
$$\pm k, \pm (3c + k), \pm (3c + 2k),$$ and the short roots are $$\pm c, \pm (c+k), \pm (2c + k).$$ Set $\ba = \ba (k).$
Since the long roots of $\Phi$ form a subsystem of type $A_2$, for which our claim has already been established, we conclude that 
$\ba$ is a finite index ideal in $R$ and that (\ref{E-Ideal}) holds for all long roots. To show that (\ref{E-Ideal}) is true for the short roots as well, we need to recall the following explicit forms of the Steinberg commutator relations that were established in (\cite{CK}, Theorem 1.1):
\begin{equation}\label{E-4}
[e_{k}(s), e_{c} (t)] = e_{c+k} (\varepsilon_1 st) e_{2c + k} (\varepsilon_2 st^2) e_{3c + k} (\varepsilon_3 st^3) e_{3c + 2k} (\varepsilon_4 s^2 t^3),
\end{equation}
\begin{equation}\label{E-5}
[e_{c+k} (s), e_{2c+k}(t)] = e_{3c+2k} (3 \varepsilon_5 st),
\end{equation}
%\begin{equation}
%[e_{3c+k}(s), e_{k} (t)] = e_{3c + 2k} (\varepsilon_6 st),
%\end{equation}
where $\varepsilon_i = \pm 1.$
Using (\ref{E-4}), we obtain
$$
[e_k(s), e_c(1)] [e_k (s), e_c(-1)] =
$$
$$
=e_{c+k} (\varepsilon_1 s) e_{2c+k} (\varepsilon_2 s) e_{3c + k} (\varepsilon_3 s) e_{3c+2k} (\varepsilon_4 s^2) e_{c+k} (-\varepsilon_1 s) e_{2c+k} (\varepsilon_2 s) e_{3c + k} (-\varepsilon_3 s) e_{3c+2k} (-\varepsilon_4 s^2).
$$
Since the terms $e_{3c+k} (-\varepsilon_3 s)$ and $e_{3c + 2k} (-\varepsilon_4 s^2)$ commute with all other terms, the last expression reduces to
$$
e_{c+k} (\varepsilon_1 s) e_{2c +k} (\varepsilon_2 s) e_{c+k}(-\varepsilon_1 s) e_{2c + k} (\varepsilon_2 s),
$$
which, using (\ref{E-5}), can be written in the form
$$
e_{3c + 2k} (3 \varepsilon_5 \varepsilon_1 \varepsilon_2 s^2) e_{2c + k} (2 \varepsilon_2 s).
$$
Hence if $s \in \ba$, we obtain that
$$
[e_k(s/2), e_c(1)] [e_k (s/2), e_c(-1)] = e_{3c + 2k} (3 \varepsilon_5 \varepsilon_1 \varepsilon_2 s^2/4) e_{2c + k} (\varepsilon_2 s) \in N.
$$
But $e_{3c + 2k} (3 \varepsilon_5 \varepsilon_1 \varepsilon_2 s^2/4) \in N,$ from which it follows that $e_{2c + k} (\mathfrak{a}) \subset N.$ This completes the proof.
\end{proof}

\vskip1mm

\noindent {\bf Remark.} If $R$ is the ring of algebraic $S$-integers, then any subgroup of finite index of the additive group of $R$ contains an ideal of finite index, so the conclusion of Proposition \ref{P-3} holds for root systems of rank $>1$ of all types without any additional restrictions on $R$. On the other hand, if $R$ is the ring of $S$-integers in a global field of positive characteristic $>2$, then $2 \in R^{\times}$, and Proposition \ref{P-3} again applies to all root systems without any extra assumptions.

\section{Proof of the main theorem}

We return to the notations introduced in \S \ref{S:I}. In particular, we set $\Gamma = E(R)$, where $R$ is a commutative noetherian ring such that $2 \in R^{\times}$ if our root system $\Phi$ is of type $C_n$ ($n \geq 2$) or $G_2$, and let $\widehat{\Gamma}$ and $\overline{\Gamma}$ denote the profinite and congruence completions of $\Gamma$, respectively. Furthermore, we let $\pi \colon \widehat{\Gamma} \to \overline{\Gamma}$ denote the canonical continuous homomorphism, so that $C(\Gamma) := \ker \pi$ is the congruence kernel. For each root $\alpha \in \Phi$, we let $\widehat{U}_{\alpha}$ and $\overline{U}_{\alpha}$ denote the closures of the images of the natural homomorphisms $U_{\alpha} (R) \to \widehat{\Gamma}$ and $U_{\alpha} (R) \to \overline{\Gamma}.$ By Proposition \ref{P-3}, the profinite and congruence topologies of $\Gamma$ induce the same topology on each $U_{\alpha} (R)$, which implies that $\pi \vert_{\widehat{U}_{\alpha}} \colon \widehat{U}_{\alpha} \to \overline{U}_{\alpha}$ is a group isomorphism. From the definitions, it is clear that $\overline{U}_{\alpha}$ coincides with $\overline{e}_{\alpha} (\widehat{R})$, where $\overline{e}_{\alpha} \colon \widehat{R} \to G(\widehat{R}) = \overline{G(R)}$ is the 1-parameter subgroup associated with $\alpha$ over the ring $\widehat{R}.$ Set
$$
\widehat{e}_{\alpha} = (\pi \vert_{\widehat{U}_{\alpha}})^{-1} \circ \overline{e}_{\alpha}.
$$
Then $\widehat{e}_{\alpha} \colon \widehat{R} \to \widehat{U}_{\alpha}$ is an isomorphism of topological groups, and in particular, we have
$$
\widehat{e}_{\alpha} (r+s) = \widehat{e}_{\alpha} (r) \widehat{e}_{\alpha} (s)
$$
for all $r, s \in \widehat{R}$ and any $\alpha \in \Phi.$

Before establishing some further properties of the $\widehat{e}_{\alpha}$, let us recall that for any commutative ring $S$ and any $\alpha, \beta \in \Phi$, $\beta \neq -\alpha$, there is a relation in $G(S)$ of the form
\begin{equation}\label{E-Steinberg}
[e_{\alpha} (s), e_{\beta} (t)] = \prod e_{i \alpha + j \beta} (N_{\alpha, \beta}^{i,j} s^i t^j)
\end{equation}
for all $s,t \in S$,
where the product is taken over all roots of the form $i \alpha + j \beta$ with $i, j \in \Z^+$, listed in an arbitrary (but {\it fixed}) order, and the $N^{i,j}_{\alpha, \beta}$ are integers depending only on  $\alpha, \beta \in \Phi$ and the order of the factors in (\ref{E-Steinberg}), but not on $s, t \in S$. Furthermore, recall that the abstract group $\tilde{G}(S)$ with generators $x_{\alpha} (s)$ for all $s \in S$ and $\alpha \in \Phi$ subject to the relations
\vskip1mm

(R1) $\tilde{x}_{\alpha}(s) \tilde{x}_{\alpha}(t) = \tilde{x}_{\alpha} (s+t)$,

\vskip1mm

(R2) \parbox[t]{15cm}{$[\tilde{x}_{\alpha} (s), \tilde{x}_{\beta} (t)] = \prod \tilde{x}_{i \alpha + j \beta} (N^{i,j}_{\alpha, \beta} s^i t^j)$, where $N_{\alpha, \beta}^{i,j}$ are the same integers, and the roots are listed in the same order, as in (\ref{E-Steinberg}),}

\vskip1mm
\noindent is called the {\it Steinberg group}. It follows from (\ref{E-Steinberg}) that there exists a canonical homomorphism $\tilde{G}(S) \to G(S)$, defined by $x_{\alpha} (s) \mapsto e_{\alpha} (s)$, whose kernel is denoted by $K_2 (\Phi, S).$
\begin{lemma}\label{L-3}
{\rm (1)} \parbox[t]{15cm}{For any $\alpha, \beta \in \Phi$, $\beta \neq -\alpha$, and $s, t \in \widehat{R}$, we have $[\widehat{e}_{\alpha} (s), \widehat{e}_{\beta} (t)] = \prod \widehat{e}_{i \alpha + j \beta} (N_{\alpha, \beta}^{i,j} s^i t^j).$}

\vskip2mm

\noindent Let $\widehat{R} = \prod_{\m \in \M} R_{\m}$ be the decomposition from Lemma \ref{L-1}(1), and for $\m \in \M$, let $\widehat{\Gamma}_{\m}$ (resp. $\widehat{\Gamma}_{\m}'$) be the subgroup of $\widehat{\Gamma}$ (algebraically) generated by $\widehat{e}_{\alpha} (r)$ for all $r \in R_{\m}$ (resp., $r \in R_{\m}' := \prod_{\n \neq \m} R_{\n}$) and all $\alpha \in \Phi$. Then

\noindent {\rm (2)} \parbox[t]{16cm}{There exists a surjective group homomorphism $\theta_{\m} \colon \tilde{G}(R_{\m}) \to \widehat{\Gamma}_{\m}$ such that $x_{\alpha} (r) \mapsto \widehat{e}_{\alpha} (r)$ for all $r \in R_{\m}$ and $\alpha \in \Phi.$}

\vskip1mm

\noindent {\rm (3)} \parbox[t]{16cm}{$\widehat{\Gamma}_{\m}$ and $\widehat{\Gamma}_{\m}'$ commute elementwise inside $\widehat{\Gamma}.$}
\end{lemma}
\begin{proof}
(1) Define two continuous maps
$$
\varphi \colon \widehat{R} \times \widehat{R} \to \widehat{\Gamma}, \ \ \ (s,t) \mapsto [\widehat{e}_{\alpha} (s), \widehat{e}_{\beta} (t)]
$$
and
$$
\psi \colon \widehat{R} \times \widehat{R} \to \widehat{\Gamma}, \ \ \ (s,t) \mapsto \prod \widehat{e}_{i \alpha + j \beta} (N_{\alpha, \beta}^{i,j} s^i t^j).
$$
It follows from (\ref{E-Steinberg}) that these maps coincide on $R \times R.$ Since $R \times R$ is dense in $\widehat{R} \times \widehat{R},$ we have $\varphi \equiv \psi,$ yielding our claim.

(2) Since we have shown that the $\widehat{e}_{\alpha}(r)$, $r \in R_{\m}$, $\alpha \in \Phi,$ satisfy the relations (R1) and (R2), the existence of the homomorphism $\theta_{\m}$ follows.

(3) It suffices to show that for any $\alpha, \beta \in \Phi$ and any $r \in R_{\m}, \ s \in R_{\m}',$ the elements $\widehat{e}_{\alpha} (r), \widehat{e}_{\beta} (s) \in \widehat{\Gamma}$ commute. Since $r s= 0$ in $\widehat{R},$ this fact immediately follows from (1) if $\beta \neq -\alpha.$ To handle the remaining case $\beta = - \alpha,$ we observe that for any ring $S$ and the corresponding Steinberg group $\tilde{G}(S)$, we have
\begin{equation}\label{E:St-1}
\tilde{G}(S) = <x_{\gamma} (r) \mid \gamma \in \Phi \setminus \{ \alpha \}, r \in S>.
\end{equation}
Indeed, it is well-known that $\tilde{G}(S)$ is generated by the elements $x_{\gamma} (r)$ for all $r \in R$ and all $\gamma$ in an arbitrarily chosen system $\Pi \subset \Phi$ of simple roots (this follows, for example, from the fact that the Weyl group of $\Phi$ is generated by the reflections corresponding to simple roots, and moreover, every root lies in the orbit of a simple root under the action of the Weyl group).
On the other hand, since $\Phi$ is of rank $\geq 2$, for any $\alpha \in \Phi,$ one can find a system of simple roots $\Pi \subset \Phi$ that does not contain $\alpha$, and (\ref{E:St-1}) follows. Using the homomorphism $\theta_{\m}$ constructed in part (2), we conclude from (\ref{E:St-1}) that $\widehat{\Gamma}_{\m} = \theta_{\m} (\tilde{G}(R_{\m}))$ is generated by $\widehat{e}_{\gamma}(r)$ for $r \in R_{\m}$, $\gamma \in \Phi \setminus \{ \alpha \}$. So, since we already know that $\widehat{e}_{-\alpha} (s)$, with $s \in R_{\m}'$, commutes with all of these elements, it also commutes with $\widehat{e}_{\alpha} (r),$ yielding our claim.

\end{proof}

The following lemma, which uses results of Stein \cite{St2} on the computation of $K_2$ over semi-local rings, is a key ingredient in the proof of the Main Theorem.
\begin{lemma}\label{L-4}
The kernel $\ker (\pi \vert_{\widehat{\Gamma}_{\m}})$ of the restriction $\pi \vert_{\widehat{\Gamma}_{\m}}$ lies in the center of $\widehat{\Gamma}_{\m}$, for any $\m \in \M.$ 
\end{lemma}
\begin{proof}
Stein has shown that if $\Phi$ has rank $\geq 2$ and $S$ is a semi-local ring which is generated by its units, then $K_2 (\Phi, S)$ lies in the center of $\tilde{G}(S)$ (cf. \cite{St2}, Theorem 2.13). Since $S = R_{\m}$ is local, it is automatically generated by its units, hence $K_2 (\Phi, R_{\m}) = \ker (\tilde{G}(R_{\m}) \stackrel{\mu}{\longrightarrow} E(R_{\m}))$ is central. On the other hand, $\mu$ admits the following factorization:
$$
\tilde{G}(R_{\m}) \stackrel{\theta_{\m}}{\longrightarrow} \widehat{\Gamma}_{\m} \stackrel{\pi \vert_{\widehat{\Gamma}_{\m}}}{\longrightarrow} E(R_{\m}).
$$
Since $\theta_{\m}$ is surjective, we conclude that
$$
\ker (\pi \vert_{\widehat{\Gamma}_{\m}}) = \theta_{\m} (K_2 (\Phi, R_{\m}))
$$
is central in $\widehat{\Gamma}_{\m}.$
\end{proof}

Now fix $\m \in \M$ and let $\Delta_{\m} = \widehat{\Gamma}_{\m} \widehat{\Gamma}_{\m}'$ be the subgroup of $\widehat{\Gamma}$ (algebraically) generated by $\widehat{\Gamma}_{\m}$ and $\widehat{\Gamma}_{\m}'.$ Let $c \in C(\Gamma) \cap \Delta_{\m},$ and write $c = c_1 c_2,$ with $c_1 \in \widehat{\Gamma}_{\m}, c_2 \in \widehat{\Gamma}_{\m}'.$
We have $\overline{\Gamma} = \overline{\Gamma}_{\m} \times \overline{\Gamma}_{\m}'$, where $\overline{\Gamma}_{\m} = E(R_{\m})$ and $\overline{\Gamma}_{\m}' = E(R_{\m}').$
Since $\pi (c_1) \in \overline{\Gamma}_{\m}$, $\pi(c_2) \in \overline{\Gamma}_{\m}',$ we conclude from $$\pi(c) = e = \pi(c_1) \pi(c_2)$$ that $\pi(c_1) = e,$ i.e. $c_1 \in \ker (\pi \vert_{\widehat{\Gamma}_{\m}}).$ Then by Lemma \ref{L-4}, $\widehat{\Gamma}_{\m}$ centralizes $c_1.$ On the other hand, $\widehat{\Gamma}_{\m}$ centralizes $c_2 \in \widehat{\Gamma}_{\m}'$ by Lemma \ref{L-3}(3). So, $\widehat{\Gamma}_{\m}$ centralizes $c.$ Thus, we have shown that $C \cap \Delta_{\m}$ is centralized by $\widehat{\Gamma}_{\m}.$ To prove that $\widehat{\Gamma}_{\m}$ actually centralizes all of $C$, we need the following
\begin{lemma}\label{L-5}
Let $\varphi \colon \mathcal{G}_1 \to \mathcal{G}_2$ be a continuous homomorphism of topological groups, and let $\mathcal{F} = \ker \varphi.$ Suppose $\Theta \subset \mathcal{G}_1$ is a dense subgroup such that there exists a compact set $\Omega \subset \Theta$ whose image $\varphi(\Omega)$ is a neighborhood of the identity in $\mathcal{G}_2.$ Then $\mathcal{F} \cap \Theta$ is dense in $\mathcal{F}.$
\end{lemma}
\begin{proof}
Since $\varphi(\Omega)$ is a neighborhood of the identity in $\mathcal{G}_2$, we can find an open set $U \subset \mathcal{G}_1$ such that
$$
\mathcal{F} \subset U \subset \varphi^{-1}(\varphi(\Omega)) = \Omega \mathcal{F}.
$$
Now since $\Theta$ is dense in $\mathcal{G}_1$, we have $U \subset \overline{\Theta \cap U},$ where the bar denotes the closure in $\mathcal{G}_1.$ Thus,
$$
\mathcal{F} \subset \overline{\Theta \cap U} \subset \overline{\Theta \cap \Omega \mathcal{F}}.
$$
But $\Theta \cap \Omega \mathcal{F} = \Omega (\Theta \cap \mathcal{F}),$ and since $\Omega$ is compact, the product $\Omega \overline{(\Theta \cap \mathcal{F})}$ is closed. So
$$
\mathcal{F} \subset \overline{\Theta \cap \Omega \mathcal{F}} \subset \Omega \overline{(\Theta \cap \mathcal{F})}.
$$
Since $\mathcal{F}$ is closed, we have $\overline{\Theta \cap \mathcal{F}} \subset \mathcal{F},$ so
$$
\mathcal{F} = (\Omega \cap \mathcal{F}) \overline{(\Theta \cap \mathcal{F})} \subset (\Theta \cap \mathcal{F}) \overline{(\Theta \cap \mathcal{F})} = \overline{\Theta \cap \mathcal{F}},
$$
as required.
\end{proof}
In order to apply Lemma \ref{L-5} in our situation, we noted the following simple fact
\begin{lemma}\label{L-6}
The subgroup $\Delta \subset \widehat{\Gamma}$ (algebraically) generated by the $\widehat{\Gamma}_{\m}$ for all $\m \in \M$ is dense. Consequently, for any $\m \in \M,$ the subgroup $\Delta_{\m} = \widehat{\Gamma}_{\m} \widehat{\Gamma}_{\m}' \subset \widehat{\Gamma}$ is dense.
\end{lemma}
\begin{proof}
Let
$$
R_0 := \sum_{\m \in \M} R_{\m} \subset \widehat{R} = \prod_{\m \in \M} R_{\m}.
$$
Clearly $R_0$ is a dense subring of $\widehat{R}.$ On the other hand, $\Delta$ obviously contains $\widehat{e}_{\alpha} (R_0)$ for any $\alpha \in \Phi.$ So, the closure $\overline{\Delta}$ contains $\widehat{e}_{\alpha} (R)$ for all $\alpha \in \Phi$, and therefore coincides with $\widehat{\Gamma},$ yielding our first assertion. Furthermore, for any $\m \in \M,$ the subgroup $\Delta_{\m}$ contains $\Gamma_{\n}$ for all $\n \in \M,$ so our second assertion follows.
\end{proof}

\vskip1mm

\noindent {\it Conclusion of the proof of the Main Theorem}: Fix $\m \in \M.$ We have already seen that $\widehat{\Gamma}_{\m}$ centralizes $C \cap \Delta_{\m}.$ We claim that $C \cap \Delta_{\m}$ is dense in $C$, and hence $\widehat{\Gamma}_{\m}$ centralizes $C$. Indeed, by Lemma~\ref{L-6}, $\Delta_{\m}$ is dense in $\widehat{\Gamma}.$ On the other hand, it follows from Corollary \ref{C-1} that there exists a string of roots $(\alpha_1, \dots, \alpha_L)$ such that the map
$$
\widehat{R}^L \to \overline{\Gamma}, \ \ \ \ \ (r_1, \dots, r_L) \mapsto \prod_{i=1}^L \overline{e}_{\alpha_i} (r_i)
$$
is surjective. Then
$$
\Omega := \widehat{e}_{\alpha_1} (\widehat{R}) \cdots \widehat{e}_{\alpha_L} (\widehat{R}) = \left( \widehat{e}_{\alpha_1} (R_{\m}) \cdots \widehat{e}_{\alpha_L} (R_{\m}) \right) \left( \widehat{e}_{\alpha_1} (R_{\m}') \cdots \widehat{e}_{\alpha_L} (R_{\m}') \right)
$$
is a compact subset of $\widehat{\Gamma}$ that is contained in $\Delta_{\m}$ and has the property that $\pi(\Omega) = \overline{\Gamma}.$ Invoking Lemma \ref{L-5}, we obtain that $C \cap \Delta_{\m}$ is dense in $C$, as required.

We now see that $\widehat{\Gamma}_{\m}$ centralizes $C$ for all $\m \in \M.$ Since the subgroup $\Delta \subset \widehat{\Gamma}$ generated by the $\widehat{\Gamma}_{\m}$ is dense in $\widehat{\Gamma}$ by Lemma \ref{L-6}, we obtain that $\widehat{\Gamma}$ centralizes $C$, completing the proof.
\hfill $\Box$

\vskip2mm

To put our proof of the Main Theorem into perspective, we recall the following criterion for the centrality of the congruence kernel in the context of the congruence subgroup problem for algebraic groups over global fields (see \cite{PR}, Theorem 4). Let $G$ be an absolutely almost simple simply connected algebraic group over a global field $K$, and $S$ be a set of places of $K$, which we assume to contain all archimedean places if $K$ is a number field, such that the corresponding $S$-arithmetic group $G(\mathcal{O}_S)$ is infinite (where $\mathcal{O}_S$ is the ring of $S$-integers in $K$). Then by the Strong Approximation Theorem, the $S$-congruence completion $\overline{G(K)}$ of the group $G(K)$ of $K$-rational points can be identified with the group of $S$-adeles $G(\mathbb{A}_S)$, and in particular the group $G(K_v)$, for $v \notin S$, can be viewed as a subgroup of $\overline{G(K)}.$ Assume furthermore that $S$ contains no nonarchimedean anisotropic places for $G$ and that $G/K$ satisfies the Margulis-Platonov conjecture. If for each $v \in S$, there exists a subgroup $H_v$ of the $S$-arithmetic completion $\widehat{G(K)}$ such that

\vskip1mm

(1) \parbox[t]{16cm}{$\pi (H_v) = G(K_v)$ for all $v \notin S$, where $\pi \colon \widehat{G(K)} \to \overline{G(K)}$ is the canonical projection;}

\vskip1mm

(2) \parbox[t]{16cm}{$H_{v_1}$ and $H_{v_2}$ commute elementwise for $v_1 \neq v_2$;}

\vskip1mm

(3) \parbox[t]{16cm}{the $H_v$, for $v \notin S$, (algebraically) generate a dense subgroup of $\widehat{G(K)}$,}

\vskip1mm

\noindent then the congruence kernel $C^S(G) := \ker \pi$ is central. So, this criterion basically states that in the arithmetic situation, the mere existence of elementwise commuting lifts of ``local groups" implies the centrality of the congruence kernel. In our situation, the existence of elementwise commuting lifts (which we denoted $\widehat{\Gamma}_{\m}$ above) also plays a part in the proof of centrality (cf. Lemma \ref{L-4}(3)), but some additional considerations (such as the result of Stein and the bounded generation property for $E(\widehat{R}) = G(\widehat{R})$) are needed; the facilitating factor in the arithmetic situation is the action of the group $G(K)$ on the congruence kernel, which is not available over more general rings. 

Finally, we will relate our result on the centrality of the congruence kernel $C(\Gamma)$ for $\Gamma = E(R)$ to the congruence subgroup problem for $G(R).$ We have the following commutative diagram induced by the natural embedding $\Gamma \hookrightarrow G(R)$:
$$
\xymatrix{1 \ar[r] & C(\Gamma) \ar[r] \ar[d]_{\alpha} & \widehat{\Gamma} \ar[d]_{\beta} \ar[r]^{\pi^{\Gamma}} & \overline{\Gamma} \ar[d]_{\gamma} \ar[r] & 1 \\ 1 \ar[r] & C(G(R)) \ar[r] & \widehat{G(R)} \ar[r]^{\pi^{G(R)}} & \overline{G(R)} \ar[r] & 1}
$$
We note that by Proposition \ref{P-2}, $\gamma$ is an isomorphism. So, $\alpha(C(\Gamma)) = C(G(R)) \cap \beta (\widehat{\Gamma})$, and $\beta (\widehat{\Gamma})$ coincides with the closure $\check{\Gamma}$ of $\Gamma$ in $\widehat{G(R)}$. Thus, our Main Theorem yields the following
\begin{cor}
$C(G(R)) \cap \check{\Gamma}$ is centralized by $\check{\Gamma}.$ 
\end{cor}
The exact relationship between $C(G(R))$ and $C(G(R)) \cap \check{\Gamma}$ (or $C(\Gamma)$) remains unclear except in a few cases. Matsumoto \cite{M1} showed that $G(R) = E(R)$ for any ring $R$ of algebraic $S$-integers, which combined with our Main Theorem and the remark at the end of \S 3, yields the centrality of $C(E(R)) = C(G(R))$, established by Matsumoto himself. Furthermore,
for $G = SL_n$ $(n \geq 3)$ and $R = \Z[x_1, \dots, x_k]$, by a result of Suslin \cite{Su}, we again have $G(R) = E(R),$ so $C(G(R)) = C(E(R))$ is central in $\widehat{E(R)} = \widehat{G(R)},$ 
which was established in \cite{KN}. On the other hand, there exist principal ideal domains $R$ for which $SL_n (R) \neq E(R)$ (cf. \cite{G}, \cite{I}), and then the analysis of $C(G(R))$ requires more effort. We only note that if $\Gamma = E(R)$ has finite index in $G(R)$, then the profinite topology on $\Gamma$ is induced by the profinite topology of $G(R)$, which implies that $\beta$ is injective, and therefore $C(\Gamma)$ is identified with a finite index subgroup of $C(G(R)).$

\vskip2mm

\noindent {\bf Acknowledgements.} The first-named author was partially supported by NSF grant DMS-0965758 and the Humboldt Foundation. The paper was finalized when both authors were visiting SFB 701 (Bielefeld), whose hospitality is gratefully acknowledged.

\bibliographystyle{amsplain}

\end{document}